\documentclass{article}
\usepackage{amsmath}
\usepackage{amssymb}
\usepackage{amsthm}
\usepackage[margin=1in]{geometry}
\usepackage{graphicx}
\usepackage{tikz}
\usepackage{url}
\usetikzlibrary{arrows.meta,positioning}

\newcommand{\R}{\mathbb{R}}
\newcommand{\Hh}{\mathbb{H}}

\newtheorem{theorem}{Theorem}[section]
\newtheorem{remark}{Remark}[section]
\newtheorem{lemma}{Lemma}[section]

\title{Estimating the number of real zeros of linear combinations of radicals of polynomials}
\usepackage{authblk}
\author[1]{Gal Binyamini\thanks{\texttt{gal.binyamini@weizmann.ac.il}}}
\author[1]{Avner Kiro\thanks{\texttt{avner-ephraiem.kiro@weizmann.ac.il}}}
\author[2]{Alexander Logunov\thanks{\texttt{alogunov@mit.edu}}}
\author[1]{Dmitry Novikov\thanks{\texttt{dmitry.novikov@weizmann.ac.il}}}
\author[2]{Dmitrii Zakharov\thanks{\texttt{zakhdm@mit.edu}}}

\affil[1]{Department of Mathematics, Weizmann Institute of Science, Rehovot, Israel}
\affil[2]{Department of Mathematics, Massachusetts Institute of Technology, Cambridge, MA, USA}

\begin{document}

\maketitle

\begin{abstract}
We obtain upper bounds for the number of real zeros of functions of the form
\[
f(x) = \sum_{k=1}^{n} c_k \bigl(P_k(x)\bigr)^{\alpha_k},
\]
where $c_k, \alpha_k \in \mathbb{R}$ and each $P_k$ is a real polynomial of degree at most $d$ that is non-negative on an interval $I\subset \mathbb{R}$. We improve previously known exponential upper bounds for the number of roots on $I$ to bounds that are polynomial in $n$, linear in $d$, and independent of the exponents $\alpha_k$.

For linear combinations of square roots of positive quadratic polynomials on $\mathbb{R}$ we prove the linear bound $2n$, answering a question of N.~Alon. A modification of the argument yields a linear bound for a question of A.~Gabrielov, D.~Novikov, and B.~Shapiro related to Maxwell’s conjecture. 

The article describes two independent approaches: an elementary ODE method in the general case, 
which also gives a polynomial bound for the number of critical points of one dimensional Gaussian mixtures,
and a PDE method for the case of positive quadratic polynomials, which connects the problem to the number of nodal domains of solutions to $\Delta u + \lambda u = 0$ on the punctured hyperbolic plane. As a byproduct of the second approach, we describe a curious relation between axially symmetric harmonic functions on $\mathbb{R}^3\setminus\{(x,0,0)\}$ and Laplace-Beltrami eigenfunctions on the hyperbolic plane with eigenvalue $1/4$.
\end{abstract}

\section{Introduction and main results}

How large can the number of distinct real roots of a function of the form

\begin{equation} \label{sum}
    f(x) = \sum_{i=1}^{n} c_i \sqrt{P_i(x)},
\end{equation}
 be, where \(P_i\) are non-negative quadratic polynomials, the coefficients \(c_i \in \mathbb{R}\), and \(f\) is not identically zero? 
 In the Spring of 2025, N.~Alon gave a talk at MIT, where he asked whether a linear upper bound holds for the number of real roots of $f$. We answer this question affirmatively by proving an upper bound $2n$ for the number of real zeroes of $f$. We also establish a linear bound in Conjecture~1.9 of Gabrielov, Novikov and Shapiro~\cite{GNS}.

This article intentionally focuses only on one dimensional problems and proposes two different methods of estimating the number of real zeroes of functions, which resemble polynomials, but we will say a few words about higher dimensional challenges.

The higher-dimensional problem (starting from dimension $d=2$) of bounding the number of isolated solutions to systems of equations involving radicals and exponentials---in the spirit of extending Bezout's theorem to non-algebraic settings--is a deep question in real algebraic geometry and fewnomial theory. We refer the interested reader to \cite{ADK, GNS, AEH}, but we highlight two notable examples of specific open problems.

Maxwell's problem asks to bound the number of equilibrium points of the Coulomb potential generated by $n$ point charges in $\mathbb{R}^3$---either assuming all charges are positive, or considering charges of mixed signs in generic position. 
Maxwell historically stated $(n-1)^2$ upper bound simply as an exercise.
Recently, an AI assisted counterexample to Maxwell's $(n-1)^2$ conjecture was announced in \cite{ABK} for $n=5$.
The article \cite{ABK} established examples of configurations of $n$ charges with at least $10n-C$ critical points of the potential.
For large $n$, we are not aware of configurations yielding more than a linear number of isolated critical points in $n$.

Meanwhile, current upper bounds remain exponential in $n$ \cite{GNS}, making the search for a polynomial upper bound in $n$ a long standing open problem. Surprisingly, in the setting of strictly positive charges, it is not even known whether the number of critical points is always finite \cite{Eremenko2015}. 

A similar challenge in establishing the finiteness of critical points arises in a second notable example: estimating the number of modes in Gaussian mixtures, see \cite{AEH}. This problem asks for an upper bound on the number of local maxima/critical points of a probability density function defined as a convex combination of $n$ multivariate Gaussians in $\mathbb{R}^d$. Explicitly, the density is given by
\begin{equation*}
    p(x) = \sum_{i=1}^n \frac{w_i}{(2\pi)^{d/2} \sqrt{\det(\Sigma_i)}} \exp\left( -\frac{1}{2} (x-\mu_i)^T \Sigma_i^{-1} (x-\mu_i) \right),
\end{equation*}
where $w_i > 0$ are mixing weights satisfying $\sum_{i=1}^n w_i = 1$, $\mu_i \in \mathbb{R}^d$ and $\Sigma_i$ are symmetric positive-definite $d \times d$ matrices. Estimating the number of the real roots of the gradient system $\nabla p(x) = 0$ is a problem similar to Maxwell's problem. Surprisingly, it is again not known whether the number of critical points is always finite for strictly positive weights $w_i > 0$. For generic covariance matrices and weights, no polynomial upper bounds in terms of  $n$ are currently known. Exponential bounds have been established and we refer the reader to \cite{AEH}, which beautifully illustrates the practical application of the methods of fewnomial theory in just a few pages.

\subsection*{Exponential upper bound and an example with $2(n-1)$ roots.}
\quad 

The statements and examples in this section are well known to specialists, see \cite{ADK}; we include them only for building intuition. We also recommend \cite{A},\cite{K} and \cite{M} for educational purposes. 

Given a sum of two square roots $\sqrt{P_1}+\sqrt{P_2}$ one can multiply it by the conjugate expression  $\sqrt{P_1}- \sqrt{P_2}$ to eliminate all square roots. This simple idea yields that a function
\[
F(x) = \prod_{\sigma \in \{-1,1\}^{n-1}} \left(c_{1}\sqrt{P_{1}(x)} + \sum_{i=2}^{n}\sigma_i c_{i}\sqrt{P_{i}(x)}\right)
\]
is a polynomial of degree at most $2^{n-1}$. The roots of $f(x)$ are a subset of the roots of this new function $F(x)$. This leads to an  upper bound of $2^{n-1}$. 

From the complex point of view, one can extend $f$ to a holomorphic function on a Riemann surface obtained from $\mathbb{C}$ by adding branch points at the roots of the $P_i$; the complex zeros of $F(z)$ correspond to the complex zeros of $f$ on this Riemann surface. In this sense, the number of complex zeros of $f$ can be exponentially large. Nevertheless, the number of real zeros admits a better upper bound, which is the subject of this article.

Below we explain a known example with at least $2(n-1)$ zeroes. Consider the linear span of the functions $\sqrt{x^{2}+k^{2}}$, $k=1,\dots,n$. These functions are linearly independent, and it is possible to find a nontrivial linear combination
\[
f(x)=\sum_{k=1}^{n} C_{k}\sqrt{x^{2}+k^{2}}
\]
such that $f(i)=0$ for $i=1,2,\dots,n-1$. Since $f(x)$ is even, it also has zeros at $i=-1,-2,\dots,-(n-1)$. This construction yields a function with at least $2(n-1)$ real zeros.

\subsection*{Main Theorems}
The main results of the article are the following.

\begin{theorem} \label{Main 1}
Let $a_k, b_k, c_k$ for $k=1, \dots, n$ be real. Then the function
\[
f(x)=\sum_{k=1}^{n}c_{k}\sqrt{(x-a_{k})^{2}+b_{k}^{2}}
\]
either has no more than $2n$ distinct real roots or is identically zero. 
\end{theorem}

To prove Theorem \ref{Main 1} one needs to know only the maximum principle for harmonic functions in Euclidean space and the fact that $1/|x|^{d-2}$ is harmonic in $\mathbb{R}^d\setminus{\{0\}}$, $d\geq 3$. The following theorem is more general and has a different proof for non-integer $\alpha$.

\begin{theorem} \label{Main 2}
Let $a_k, b_k$ and $c_k$ be real numbers.  Let $\alpha \in (-\infty, -1/2]$. If the function
\[
f(x)=\sum_{k=1}^{n}c_{k}[(x-a_{k})^{2}+b_{k}^{2}]^{\alpha}
\]
is not identically zero, then it has no more than $2(n-1)$ distinct real roots on $\R$.
\end{theorem}

Theorem~\ref{Main 2} for $\alpha=-1/2$  implies that, for any Coulomb  potential generated by $n$ point charges in $\mathbb{R}^3$, the intersection of the zero set of the potential with any line is either the whole line or has at most $2(n-1)$ points.
If such a Coulomb potential is not identically zero on a line, we also show the derivative of the potential along the line vanishes at no more than $4n-2$ points on the line, which establishes a linear bound for Conjecture 1.9 from \cite{GNS}, see Appendix.
\begin{remark}
    While this paper was in the final stages of preparation, we became aware of a preprint \cite{Oliveira}, in which the sharp bound of $2n-1$ on the number of critical points was established by different methods.
\end{remark}

\begin{remark}
The restriction on the exponent, $\alpha \le -1/2$, is connected to the spectral properties of the Laplace-Beltrami operator on the hyperbolic plane $\Hh:$ $\lambda_1(\Hh)=1/4$. 

 The case $\alpha>1/2$ appears to be more difficult and we don't know how to obtain sharp results in this case. We describe a related question on the number of nodal domains of solutions to $\Delta u + \lambda u=0$ on the Lobachevsky plane in Section \ref{Sec:Proof}.
\end{remark}

The following theorem is more flexible and provides an estimate for the number of real zeroes of a linear combination of radicals of polynomials of arbitrary degree. The estimate is weaker, but the result is quite general.

 \begin{theorem} \label{Main 3}
Let $P_k$, $k=1,\dots,n$ be real polynomials of one variable of degree at most $d$ and assume that each $P_k$ is positive on an interval $I=(a,b)$, where $a,b \in \mathbb{R}\cup\{+\infty,-\infty\}$. Let $\alpha_k$ be arbitrary real numbers. If $P_k^{\alpha_k}$ are linearly independent on $I$, $c_k$ are real and not all zero, then $$f=\sum_{k=1}^{n} c_k P_k^{\alpha_k}$$
has no more than $4n^3d+n$ distinct roots on $I$.

\end{theorem}
 
We would like to emphasize that the latter bound for the number of real roots does not depend on the choice of $\alpha_k$, which is in line with the philosophy of the fewnomial theory \cite{K}.

Assuming that the polynomials $P_k$ are generic and that none of $\alpha_k$ is an integer, we claim that the functions $P_k^{\alpha_k}$ are linearly independent. This can be seen by considering the holomorphic extension of $P_k^{\alpha_k}$:
in the generic case the $P_k$ have different roots and $P_1^{\alpha_1}(z)$ will have branching singularities only at zeroes of $P_1$, while linear combinations of $P_k^{\alpha_k}(z)$, $k=2,\dots, n$, will not have branching singularities at zeroes of $P_1$ and therefore $P_1^{\alpha_1}$ cannot be written as a linear combination of $P_k^{\alpha_k}$, $k>1$.

\section*{Acknowledgements.}
 This work was completed during the time A.L. served as Packard Fellow. A. L. is grateful for the hospitality of the School of Mathematical Sciences at Tel Aviv University, where A. L. was a visitor in January 2026 and presented the proof of these results at the analysis seminar.

D.Z. was supported by the Simons Dissertation Fellowship.

This work was partially done while D.N  and G.B. were at the Institute for Advanced Study in Princeton, and they would like to thank the institute for its hospitality and for providing excellent working conditions. G.B. was supported by the Marvin V. and Beverly J. Mielke Endowed Fund and the Infosys Member Fund, and D.N. was supported by the Kovner Member Fund. G.B. and A.K. were also supported by the European Union (ERC, SharpOS, 101087910) and by the Israel Science Foundation (grant No. 2067/23). D.N. was also supported by the Israel Science Foundation grant 1167/17 and by Minerva grant 714141.

Declaration of Generative AI in Scientific Writing: During the preparation of this manuscript, the authors utilized AI tools strictly for editorial support (correcting misprints, language polishing, and finding references). All content, proofs, and mathematical arguments remain entirely the authors' own responsibility.

\section{Proof of Theorem \ref{Main 3} via ODE}

\textbf{Claim.} If a function $f=P^{\alpha}$, where $P$ is a non-zero polynomial of degree $d$, then
$f^{(k)}/f$ is a rational function and $f^{(k)}/f=\frac{Q}{P^k}$, where $Q$ has degree at most $k(d-1)$.

The latter statement can be proved by induction on $k$. For $k=1$ we have $f'/f= \frac{\alpha P'}{P}$.
The induction step from $k$ to $k+1$ can be verified in a straightforward manner: if $f^{(k)}=f \frac{Q}{P^k}= P^\alpha \frac{Q}{P^k},$ then $$ \quad f^{(k+1)}= (f^{(k)})'= (P^\alpha)' \frac{Q}{P^k}+ P^\alpha \left(\frac{Q}{P^k}\right)'= P^\alpha \frac{\alpha P'Q+Q'P-kP'Q}{P^{k+1}}= f \frac{(\alpha-k)P'Q+Q'P}{P^{k+1}},$$
and the degree of the numerator $(\alpha-k)P'Q+Q'P$ is at most $(k+1)(d-1)$.

Denote $P_k^{\alpha_k}$ by $f_k$. Put $W_0=1$ and consider the Wronskian of $f_1,\dots,f_k$:
\begin{equation}
    W_k = \begin{vmatrix}
f_1 & f_2 & \cdots & f_k \\
f_1' & f_2' & \cdots & f_k' \\
f_1'' & f_2'' & \cdots & f_k'' \\
\vdots & \vdots & \ddots & \vdots \\
f_1^{(k-1)} & f_2^{(k-1)} & \cdots & f_k^{(k-1)}
\end{vmatrix}
\end{equation}

We know that $W_k$ is not identically zero on $I$ because $P_k^{\alpha_k}$ are real-analytic and linearly independent.

\textbf{Claim.} The Wronskian $W_k$ has no more than $k^2d$ roots on $I$.

To prove the claim we write 
$W_k$ as

\begin{equation}
 \begin{vmatrix}
f_1 & f_2 & \cdots & f_k \\
f_1' & f_2' & \cdots & f_k' \\
f_1'' & f_2'' & \cdots & f_k'' \\
\vdots & \vdots & \ddots & \vdots \\
f_1^{(k-1)} & f_2^{(k-1)} & \cdots & f_k^{(k-1)}
\end{vmatrix} =  \begin{vmatrix}
1 & 1 & \cdots & 1 \\
f_1'/f_1 & f_2'/f_2 & \cdots & f_k'/f_k \\
f_1''/f_1 & f_2''/f_2 & \cdots & f_k''/f_k \\
\vdots & \vdots & \ddots & \vdots \\
f_1^{(k-1)}/f_1 & f_2^{(k-1)}/f_2 & \cdots & f_k^{(k-1)}/f_k
\end{vmatrix} \cdot \prod f_i. 
\end{equation}
We can divide the Wronskian by $\prod f_i$, which do not have roots on $I$, and then multiply the columns of $W_k$ by $[P_j]^{k-1}$ to conclude that $$\tilde{W_k}= \frac{W_k}{\prod f_i} \prod_{j=1}^{k} [P_j]^{k-1} $$ is a determinant of a matrix, whose entries are polynomials with degree at most $(k-1)d$. Hence $\tilde{W_k}$ has degree at most $k^2d$ and $W_k$ has no more than $k^2d$ roots on $I$. In the estimate of the number of zeroes of $W_k$ we intentionally wrote crude (non-sharp) bounds to keep the formulas shorter.

Now, consider any non-zero linear combination $f=\sum_{k=1}^{n} c_k f_k$.
We would like to state the Frobenius formula: \(f \) satisfies an ODE:
\begin{equation}
\frac{\partial}{\partial x}\frac{W_{n-1}^2}{W_n W_{n-2}}\frac{\partial}{\partial x} \cdots \frac{\partial}{\partial x}  \frac{W_{1}^2}{W_2 W_{0}}  \frac{\partial}{\partial x} \frac{W_{0}}{W_{1}}  f = 0.
\end{equation}
We refer to \cite{Ke} and \cite{NY} for a short proof of the Frobenius formula, which is usually stated in the context of solutions to ODEs, but this reference comes with a note that the Wronskian equation $W(f,f_1,\dots,f_n)=0$ is an ODE for $f$ with solutions $f_1,\dots f_n$.

Armed with the  Frobenius formula, we are ready to finish the proof. For brevity, by the \emph{poles} of a function $g$ on $I$ we mean the points of $I$ where $g$ is not defined; outside its poles, $g$ is assumed to be real-analytic. By the \emph{zeroes} of $g$ we mean isolated zeroes, i.e., points where $g$ vanishes but is nonzero in a punctured neighborhood of each such point.
If we take a function $g$, which has a finite number of poles on $I$, and if we perform two operations: multiplication by $\frac{W_{k-1}^2}{W_k W_{k-2}}$ and differentiation, then we can control the number of new poles from above and the number of zeroes from below. 
Multiplication by $\frac{W_{k-1}^2}{W_k W_{k-2}}$ can create at most $2k^2d$ new poles at the zeroes of $W_k W_{k-2}$ and it does not decrease the number of zeroes of $g$ by more than $2k^2d$ (the decrease happens when the new pole coincides with the old zero). If a non-zero function $g$ has at most $A$ poles on $I$ and at least $N>A+1$ zeroes on $I$, then the derivative of $g$ is a non-zero function and $g'$ has at least $N-A-1$ zeroes by Rolle's theorem.  Iterating the estimates of the last three sentences in a crude way we conclude
that $f$ cannot have more than $4n^3d+n$ zeroes; otherwise, the LHS of the Frobenius ODE for $f$ would not be zero.  
\section{What was really used in the proof with applications to exponential sums and Gaussian mixtures.}
A curious reader looking at the proof of Theorem \ref{Main 3} may wonder what was really used about the functions $f_1,f_2\dots f_n$ to guarantee a bound on the number of zeroes of their non-zero linear combination $f$. The number of zeroes of $f$ is controlled by the number of zeroes of the Wronskians $W_k$. Assuming that all $f_k$ are positive on the considered interval, the modified Wronskian $$\widetilde W_k= \frac{W_k}{f_1 \dots f_k}=\begin{vmatrix}
1 & 1 & \cdots & 1 \\
f_1'/f_1 & f_2'/f_2 & \cdots & f_k'/f_k \\
f_1''/f_1 & f_2''/f_2 & \cdots & f_k''/f_k \\
\vdots & \vdots & \ddots & \vdots \\
f_1^{(k-1)}/f_1 & f_2^{(k-1)}/f_2 & \cdots & f_k^{(k-1)}/f_k
\end{vmatrix}$$ has the same number of zeroes as $W_k$, but sometimes it happens to be a rational function. And the degrees of the rational functions $W_k$ give a bound on the number of zeroes of $f$.

We would like to mention one more situation when $\widetilde W_k$ appears to be a rational function. 
Let $f_k=P_k\exp(Q_k)$, where $P_k$ and $Q_k$ are polynomials, and let $f$ be a linear combination of $f_1,f_2\dots f_n$, which is not identically zero. Then $\widetilde W_k$ is also a rational function. Indeed, if  a function $f$ has a form  $P\exp(Q)$, then the $l$-th derivative of $f$ also has the form $\tilde P \exp(Q)$ for some non-zero polynomial $\tilde P$, but with the same $Q$. Hence all of the entries of $\widetilde W_n$, which are $f^{(a)}_b/f_b$, appear to be rational functions.

One dimensional case of the problem of estimating the number of modes of Gaussian mixtures can be attacked by this method. Namely, if $$f=\sum_{k=1}^{n} c_k e^{a_k(x-b_k)^2}, \quad a_k,b_k,\in \mathbb{R}, \quad c_k >0, $$
then the number of critical points of $f$ is no more than $Cn^3$ for some numerical $C$. We leave this statement without proof, which is a minor modification of the proof of Theorem \ref{Main 3}, but the idea was explained above. We refer to \cite{AEH} for the review of the problem of estimating the number of modes of Gaussian mixtures, which has a similar flavor to Maxwell's problem.  

 A particularly curious reader, eager to test the limits of the method of combining Rolle's theorem with Frobenious formula, might naturally consider a linear combination of double exponentials of polynomials: $$f(x)=\sum_{k=1}^{n} c_k \exp(\exp(P_k(x)),$$
and attempt to estimate its number of real zeros -- assuming that f is not identically zero.
Lest anyone mistake this for a routine exercise in estimating the degree of a rational function—particularly since $\widetilde{W}_k$ has the decency not to be rational in this case, but one may resort to the kettle method of reducing the unsolved case to a solved one.

\section{Reduction of Theorem \ref{Main 1} to the case $\alpha=-3/2$ of Theorem \ref{Main 2}}

Suppose that the function $f(x)=\sum_{k=1}^{n}c_{k}\sqrt{(x-a_{k})^{2}+b_{k}^{2}}$ has at least $2n+1$ distinct real roots. By Rolle's Theorem,  its second derivative $f''(x)$ must have at least $2n-1$ roots. A direct calculation of the second derivative yields:
\[
f''(x)=\sum_{k=1}^{n}c_{k}b_{k}^{2}\frac{1}{[(x-a_{k})^{2}+b_{k}^{2}]^{\frac{3}{2}}}
\]
This expression is a function of the type described in Theorem 2, with new coefficients $c_k' = c_k b_k^2$ and $\alpha = -3/2$. Thus  $f''(x)$ can have at most $2(n-1)$ roots. Arguing by contradiction we conclude that $f$ cannot have more than $2n$ roots.

\section{Proof of Theorem \ref{Main 2} for the case of negative half integers via maximum principle for harmonic functions.} \label{Sec: half integeres}
The method of the proof suggested in this section works only for the cases where the exponent $\alpha$ is a negative half-integer, such as $\alpha = -1/2, -1, -3/2,$ etc. 
This method is based on the fact that the function $|x|^{-(d-2)}$ is harmonic in $\R^d \setminus \{0\}$, $d\geq 3$,
and harmonic functions satisfy the maximum principle.
We will explain only the case $\alpha = -1/2$ via harmonic functions in $\mathbb{R}^3$ as the rest of the half-integer cases are proved in a similar fashion. A different proof for real $\alpha<-1/2$ is outlined in the next section.
\subsection{Axially symmetric harmonic functions and the number of nodal domains.}

\medskip

Given
\[
  f(x)=\sum_{i=1}^n c_i\Bigl((x-a_i)^2+b_i^2\Bigr)^{-1/2}
\]
and $\theta \in [0,2\pi)$ consider the function
\[
  u_\theta(x,y,z)=\sum_{i=1}^n \frac{c_i}{\sqrt{(x-a_i)^2+(y-b_i\cos \theta)^2+(z-b_i\sin \theta)^2}}.
\]
Then $u_\theta$ is harmonic in $\R^3\setminus\{(a_i,b_i\cos \theta,b_i\sin \theta)\}_{i=1}^n$, and
$u_\theta(x,0,0)=f(x)$.

Define the axially symmetric harmonic function
\[
  u(x,y,z)=\frac{1}{2\pi}\int_{0}^{2\pi} u_\theta(x,y,z)\,d\theta,
\]
which is harmonic on $\R^3\setminus\bigcup_{i=1}^n S_i$, where
\[
  S_i=\{(a_i,\,b_i\cos\theta,\,b_i\sin\theta):\ \theta\in[0,2\pi]\}.
\]
In other words, the function $u=\mu * \frac{1}{|x|}$ is a Newtonian potential in $\mathbb{R}^3$ of a signed measure $\mu$ supported on the union of $S_i$,
whose restriction on each $S_i$ is uniform on $S_i$. It is not difficult to check that the Newtonian potential of a length measure on a circle tends to infinity near the circle (with logarithmic speed) and therefore
\begin{align*}
  u(p)&\longrightarrow +\infty \ \text{as } p\to S_i \text{ if } c_i>0,\\
  u(p)&\longrightarrow -\infty \ \text{as } p\to S_i \text{ if } c_i<0.
\end{align*}

In this section by nodal domains of $u$ we will call the connected components of the complement of the zero set of $u$ in $\mathbb{R}^3$. 
The function $u$ does not change sign in each of the nodal domains and we allow the nodal domains to contain the circles, where $u$ is $+\infty$ or $-\infty$.
Note that $u(x,y,z)\longrightarrow 0 \quad \text{as } x^2+y^2+z^2\to+\infty$.
Together with the maximum principle for harmonic functions this implies that each nodal domain of $u$
must contain at least one of the circles $S_i$ and $S_i$ do not intersect the boundaries of nodal domains. We conclude that the number of nodal domains
of $u$ is at most $n$. 

\medskip

Because $u$ is axially symmetric, we can write $u(x,y,z)=h(x,r)$ where $r:=\sqrt{y^2+z^2}$. The number of nodal domains of $h$ in $\mathbb{R}^2_+=\{(x,r): r>0 \}$ is the same as the number of nodal domains of $u$ in $\mathbb{R}^3$, which is not greater than $n$, see Figure \ref{fig:zero set}.

\begin{figure}[h!]
    \centering
    \includegraphics[width=0.8\textwidth]{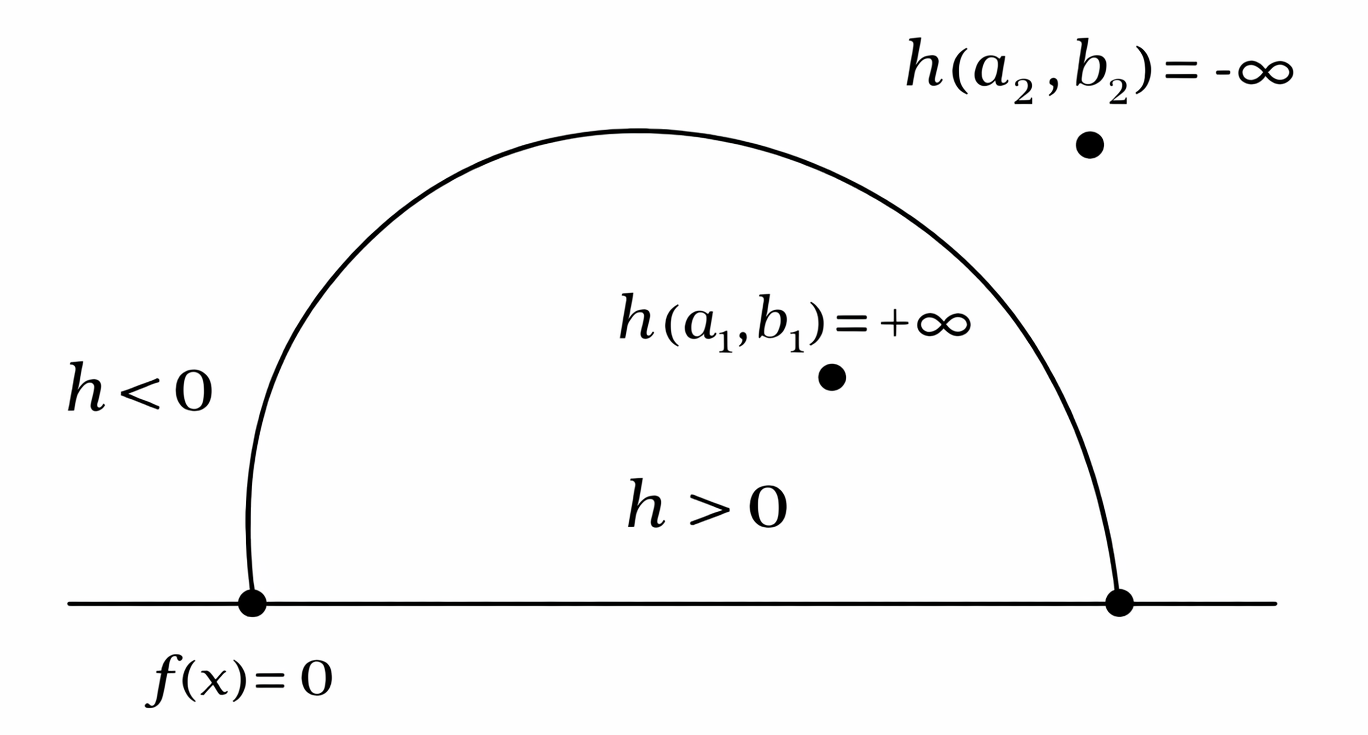}
    \caption{Zero set of $h$ with two nodal domains.}
    \label{fig:zero set}
\end{figure}

By the real-analyticity of $u$ the function $h$ has an even (in $r$) real-analytic extension  across the line $r=0$: $$h(x,r)=h(x,-r),$$ and it satisfies
\[
  h(x,0)=u(x,0,0)=f(x),\qquad x\in\R.
\]
So it is sufficient to estimate the number of zeroes of $h$ on the axis $\{r=0\}$.

\subsection{Generic case.}
To explain the intuition of the proof, we will first make the generic assumption that the zero set of $u$ has no points $p$ such that $u(p)=|\nabla u (p)|=0$. Such points are called singular points of the zero set. In this case the zero set of $h$ in $\mathbb{R}^2_+$ is a union of smooth curves. The number of zero curves of $h$ cannot be more than $n-1$, otherwise the zero curves of $h$ would separate $\mathbb{R}^2_+$ into at least $n+1$ connected components, while we know that the number of nodal domains of $h$ is at most $n$.
 Because of the real-analyticity, each zero curve of $h$ in $\R^2_{+}$ has at most 2 limit points on the line $\{(x,0)\}$.

By assumption, if $u(x,0,0)=0$, then $\nabla u(x,0,0)$ is not zero and by the axial symmetry of $u$ the zero set of $u$ is a surface perpendicular at $x$ to the line $\{(t,0,0)\}$.
Hence, any zero point of $h$ on $\{(x,0)\}$ is a limit point of a zero curve of $h$ in $\R^2_+$. Thus in the generic case the number of zeroes of $h$ is bounded by $2(n-1)$. The following topological fact was used and will be used again without a proof.

\subsection{Topological fact.}
Let $\mathbb{H}$ be the upper half-plane (or, alternatively, let $\mathbb{D}$ be the unit disc).
Suppose $\mathbb{H}$ (respectively $\mathbb{D}$) is partitioned by a collection $\Gamma$ of $C^\infty$-smooth disjoint curves into
finitely many connected components.  Each curve in $\Gamma$ is either a simple closed loop contained in $\mathbb{H}$,
    or a smooth embedding of an open interval whose two ends both tend to the boundary
    $\partial\mathbb{H}$ (respectively $\partial \mathbb{D}$);
 distinct curves in $\Gamma$ have no common points in $\mathbb{H}$ (respectively in $\mathbb{D}$),
    although they may have common limit points on the boundary $\partial\mathbb{H}$
    (respectively $\partial \mathbb{D}$).

Let $N$ be the number of connected components of the complement of 
$
\mathbb{H} \setminus \bigcup_{\gamma \in \Gamma} \gamma
\quad
\text{(respectively } 
\mathbb{D} \setminus \bigcup_{\gamma \in \Gamma} \gamma\text{).}
$
Then the number of curves in $\Gamma$ is finite and satisfies
$\#\Gamma \leq N - 1.$

\subsection{Equation for axially symmetric harmonic functions.} 

The harmonicity of the axially symmetric function $u$ in $\mathbb{R}^3\setminus\{(x_1,0,0)\}$ can be expressed in cylindrical coordinates $(x,r)$, where $r=\sqrt{y^2+z^2}$. Namely, the function $h(x,r)=u(x,y,z)$ satisfies the Darboux equation (sometimes also called the elliptic Euler–Poisson–Darboux equation):

\begin{equation}\label{eq:hPDE}
  h_{xx}+h_{rr}+\frac{1}{r}\,h_r=0.
\end{equation}
Similarly, an axially symmetric function $u(x_1,x_2,\dots, x_n)=h(x_1, \sqrt{x_2^2+\dots+ x_n^2})$ is harmonic in $\mathbb{R}^n\setminus\{(x_1,0,0 \dots):x_1\in\mathbb{R}\}$ if and only if 
 \begin{equation}
  h_{xx}+h_{rr}+\frac{n-2}{r}\,h_r=0
  \qquad \text{in }   \qquad \R^2_{+}.
\end{equation}

\subsection{Bers' theorem on equiangular intersection.}
Given a $C^\infty$-smooth function $u$ in a planar domain $\Omega\subset\mathbb{R}^2$, the singular set of the zero set of $u$ is a set of points where $u$ and $\nabla u$ are simultaneously zero. Outside the singular set the zero set of $u$ is a union of smooth curves by implicit function theorem. Bers' equiangular theorem explains how the zero sets of two-dimensional solutions to elliptic partial differential equations look like at singular points. We will formulate the theorem only in the real-analytic case, where the proof is easier and sketched below.

\textbf{Bers' equiangular theorem}(folklore). Let $u$ be a real solution to
$$\Delta u + Au_x+Bu_y+Cu =0$$
in a planar domain $\Omega\subset\mathbb{R}^2$, where the
functions $A,B,C$ are real-analytic in $\Omega\subset\mathbb{R}^2$. 
Assume that $0\in \Omega$ is a singular point of the zero set of $u$ and
the vanishing order of $u$ at $0$ is $k\geq 2$. Then, in a punctured neighborhood of $0$, the zero set is a union of $2k$ smooth curves, which  meet at $0$ in an equiangular way so that the curves have tangential directions at $0$ and the angles between adjacent curves are $\pi/k$.

\begin{remark}
 One can extend the zero curves of $u$ through $0$ in a $C^1$-smooth way. In this sense the zero set of $u$ is an equiangular intersection of $k$ $C^1$-smooth curves.
In particular, if the vanishing order of $u$ at some point $p\in \Omega$ is 2, then there will be two zero curves intersecting orthogonally at $p$.
    
\end{remark}
\begin{proof} The solutions to linear elliptic PDE with real-analytic coefficients are real-analytic (see \cite{Hormander}, pages 178-180). So we can expand 
$u$ into Taylor series at $0$, which converges absolutely in some neighborhood of $0$:
$$ u(x,y)= p_k(x,y) + \sum_{|\alpha|>k} \frac{\partial^\alpha u(0)}{\alpha!}x^{\alpha_1}y^{\alpha_2},$$
where $p_k(x,y)$ is a homogeneous polynomial of degree $k$.
Using real-analyticity, we can plug the Taylor expansion into the equation $\Delta u + Au_x+Bu_y+Cu =0$ to see that
$$\Delta p_k(z)= O(|z|^{k-1}), $$
where we identify complex numbers $z=x+iy$ with points in $\mathbb{R}^2$.
However, $\Delta p_k(z)$ is homogeneous of degree $k-2$, which implies that $\Delta p_k=0$. The space of real homogeneous harmonic polynomials of degree $k$ is two-dimensional and is spanned by the real and imaginary parts of $z^k$.
In polar coordinates, 
$$p_k(r,\varphi)= c_0 r^k \sin(k(\varphi-\varphi_0))$$
for some $c_0>0$ and $\varphi_0 \in [0,2\pi)$.
The gradient of $p_k$ vanishes only at $0$ and the Taylor expansion yields
$$ \nabla u = \nabla p_k + o(|\nabla p_k|).$$
This approximation for $u$ and its gradient implies that the zero curves of $u$ are asymptotically tangent at $0$ to the rays $\{\varphi=\varphi_0+j\pi/k\}$, $j=0,1,\dots, 2k-1$, which form the zero set of $p_k$.
    
\end{proof}
\begin{remark} \label{Rem: Bers'} If a function $u$ is $C^1$-smooth in a punctured neighborhood of zero
and $$u(z)=\Re(c_0z^{-k}), \quad \nabla u(z)=\nabla\Re( c_0z^{-k})+o(|z|^{-k-1}),$$ 
where $k$ is an integer, $c_0 \in \mathbb{C}\setminus\{0\}$, then
the zero set of $u$ is formed by $2k$ curves ending at $0$ such that the angles between adjacent curves are $\pi/k$. This remark is not used in the proof of Theorem \ref{Main 2}, but it can be used to extend the Bers' equiangular theorem to the case of poles of real solutions to $\Delta u + \lambda u=0$
on a punctured hyperbolic plane.
\end{remark}

\subsection{General case.}

We proceed to prove the bound for the number of zero points in the general case. We may assume that the zero set of $u$ on the $x$-axis is locally finite. Since the zero set of $u$ does not contain isolated points, any zero point of $h(x,0)$ is a limit point 
of the zero set of $h$ in $\R^2_{+}$.

In view of the equiangular intersection property of the zero set, see Figure \ref{fig:perturbation},
one can infinitesimally perturb the zero set of $h$ in $\R^2_+$ near the singular points in such a way that the set becomes a disjoint union of smooth curves in $\R^2_{+}$, the number of nodal domains does not increase and the set of limit points on $\{(x,0)\}$ remains the same. Hence the infinitesimal perturbation of the zero set of $h$ has no more than $(n-1)$ smooth curves and therefore the zero set of $h$ in $\R^2_{+}$ is a union of at most $n-1$ piecewise smooth curves, which intersect only at singular points. Since $h$ is real-analytic every zero curve of $h$ can't have more than $2$ limit points on the line $\{(x,0)\} $, and we conclude that the total number of zeroes is bounded by $2(n-1)$.

\begin{figure}[h!]
    \centering
    \includegraphics[width=0.8\textwidth]{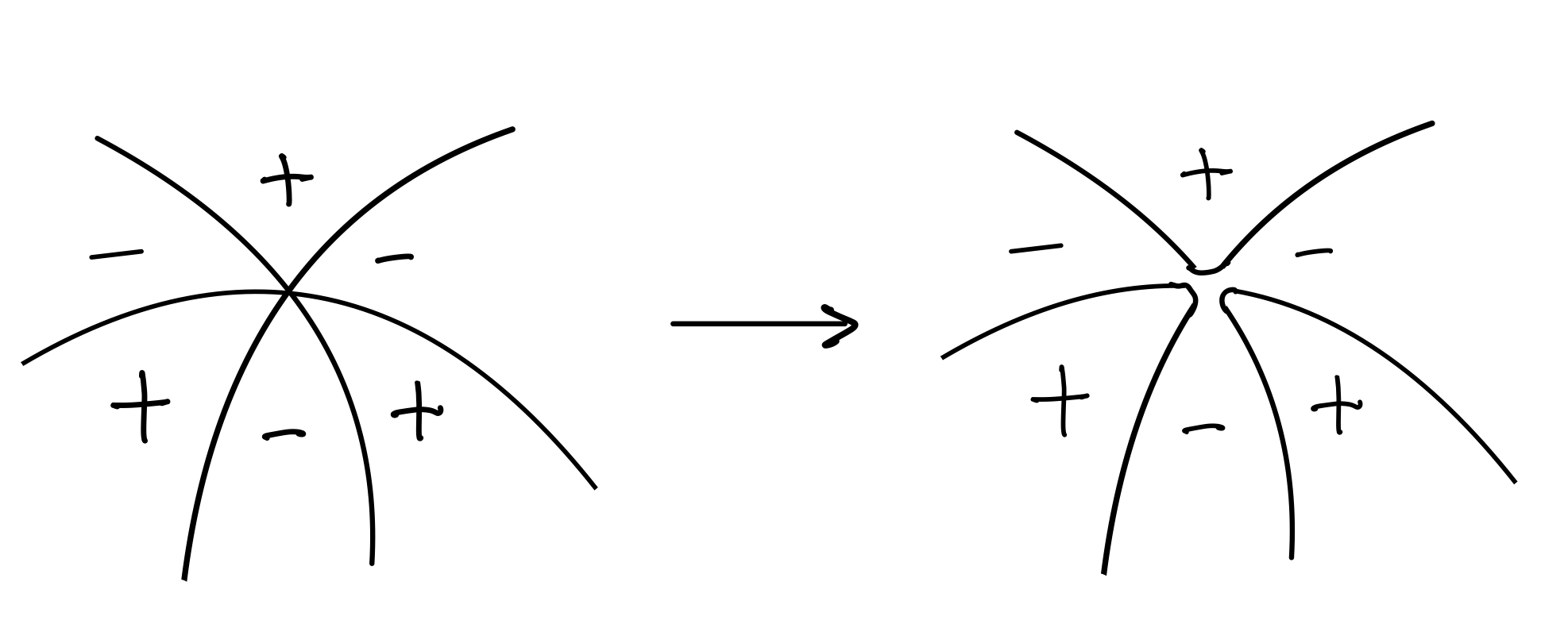}
    \caption{Infinitesimal perturbation of equiangular intersection that does not increase the number of nodal domains}
    \label{fig:perturbation}
\end{figure}

\section{Proof of Theorem \ref{Main 2} via Nodal Geometry on the hyperbolic plane} \label{Sec:Proof}
We recommend \cite{LMP} as an introduction to spectral geometry. In this section we present the proof of Theorem~\ref{Main 2} using fundamental solutions for
$\Delta + \lambda$ on the Lobachevsky plane. At first glance, these objects may appear unrelated to the original problem. For readers interested in understanding this choice, we show in the appendix how the radial fundamental solution of
$\Delta u + \frac{1}{4}u = \delta_0$ is connected to the Newtonian potential in $\mathbb{R}^3$ generated by a uniform charge distributed on a circular wire. 

\subsection*{Radial Fundamental solutions to $\Delta + \lambda $ on $\mathbb{H}$ for $\lambda <1/4$.}

We will use two standard models of the Lobachevsky (hyperbolic) plane:

\begin{itemize}
    \item \textbf{Upper half–plane model.}
    Consider the upper half–plane
    \[
    \mathbb H = \{(x,y)\in\mathbb R^2 : y>0\}
    \]
    equipped with the Poincaré metric
    \[
    ds^2 = \frac{dx^2 + dy^2}{y^2},
    \]
    and the corresponding Laplace–Beltrami operator
    \[
    \Delta_{\mathbb H} u
      = y^{2}\bigl(\partial_x^{2} + \partial_y^{2}\bigr)u.
    \]

    \item \textbf{Disc model.}
    Consider the unit disc
    \[
    \mathbb{D} = \{z = x + iy \in \mathbb{C} : |z|<1\},
    \]
    endowed with the Poincaré metric
    \[
    ds^2 = \frac{4\,|dz|^2}{(1 - |z|^2)^2}.
    \]
    The Laplace–Beltrami operator on $\mathbb{D}$ is given by
    \[
    \Delta_{\mathbb{D}} u
      = \frac{(1 - |z|^2)^2}{4}\,\Delta_{\mathbb{R}^2} u,
    \]
    where $\Delta_{\mathbb{R}^2} = \partial_x^2 + \partial_y^2$ is the Euclidean Laplacian.
\end{itemize}

We will call a function $u$ radial on $\mathbb{D}$ with respect to point $z_0$ (in the Poincare metric) if
$u(z)=u(d_\mathbb{D}(z,z_0))$ is a function of distance to $z_0$ in Poincare metric. Unless specified otherwise, radial means radial with respect to $0$.
For any real constant $\lambda$ the radial solutions to $\Delta u + \lambda u = 0$ in $\mathbb{D}\setminus{\{0\}}$ satisfy an ODE
\[
u_{rr} + \frac{1}{r}u_{r} + \frac{4\lambda}{(1-r^2)^2}u = 0,
\quad 0<r<1.
\]
We restrict attention to the case $\lambda < \frac{1}{4}$, for which there exist two positive, linearly independent radial solutions. Although there are many explicit formulas for these solutions in terms of hypergeometric functions, Legendre functions, heat kernels, and related special functions, the only properties we will use are their asymptotic behaviors as $|z| \to 1^-$ and $|z| \to 0^+$. These asymptotics are used in the proof of Theorem~\ref{Main 2}. Rather than referring to external sources, we include explicit formulas for the radial solutions and derive the relevant asymptotics in the appendix. 
\begin{lemma} \label{Lemma:asymptotics}
Let $\beta>-\frac{1}{2}$ and let $d\mu$ be the length measure on $\partial \mathbb{D}$. Then
\begin{itemize}
    \item  \begin{equation} \label{eq: radial regular}
u(z)
:= \frac{1}{2\pi}\int_{\partial\mathbb{D}}
\left(\frac{|z-\zeta|^2}{1-|z|^2}\right)^{\beta} \, d\mu(\zeta)
\end{equation}
is the unique radial solution to $\Delta_\mathbb{D} u-\beta(\beta+1)u = 0$ in $\mathbb{D}$ such that $u(0)=1$. As $|z|\to 1-$ the function $u$ has the following asymptotic properties (for $\beta > -1/2$):
\[
u(z)
= C_1(\beta)\,(1 - |z|^2)^{-\beta}\bigl(1+o(1)\bigr),
\]
where $C_1(\beta)=\frac{1}{2\pi}\int_{\partial\mathbb{D}}
|1-\zeta|^{2\beta} \, d\mu(\zeta)$ is a positive numerical constant.
   
    \item A function 
    \begin{equation} \label{eq: radial fundamental}
    g(z)= g(|z|)= \frac{u(|z|)}{2\pi}\int
    \limits_{|z|}^{1}\frac{1}{u^2(r)r}dr
    \end{equation}
      is radial fundamental solution to
    $\Delta_\mathbb{D} g+\lambda g = -\delta_0$
    in $\mathbb{D}$ (see \eqref{eq:distributions} in Appendix for the precise meaning).  As $|z|\to +0$ $$g(z)= \frac{-\log|z|}{2\pi}\bigl(1+o(1)\bigr)$$
    and as $|z|\to 1-$ 
\[
g(z)
= C_2(\beta)\,(1 - |z|^2)^{\beta+1}\bigl(1+o(1)\bigr), \quad C_2(\beta) >0.
\]
    
\end{itemize}
\end{lemma}
The proof of the lemma is given in Appendix. The representation \eqref{eq: radial regular} is true for $\beta=-1/2$, but  $u(|z|)$ does not behave like $(1-|z|^2)^{1/2}$ as $z\to 1-$.  Instead of \eqref{eq: radial fundamental} we provide a modified formula for fundamental solution
for $\Delta_{\mathbb{D}} u+\frac{1}{4}u=\delta_0$ in the appendix using the Newtonian potential of a uniform
charge on a circular wire.

The next lemma provides a bound on the number of nodal domains of linear combinations of Möbius transformations of $g$, which will be used in the proof of Theorem~\ref{Main 2}. A reader not familiar with the theory of distributions may safely ignore the term “fundamental solution” and simply assume below that the function $u$ is a linear combination of Möbius transformations of $g$.

\begin{lemma} \label{Lemma:main}
Let $\beta>-1/2$ and $\lambda=-\beta(\beta+1)< \frac{1}{4}$. Let $u$ be a solution to
\[
\Delta_{\mathbb{D}} u + \lambda u=\sum_{k=1}^{n}c_{k}\delta_{p_{k}} \quad \text{ in Poincare's disc } \quad \mathbb{D},
\]
where $c_{k} \in \mathbb{R}\setminus{0}$ and $p_k$ are distinct points on $D$. Assume that $u(z) = o(1-|z|^2)^{-\beta}$ as $z \to \partial \mathbb{D}$. Then $u$ has at most $n$ nodal domains in $D$.
\end{lemma}

\begin{proof}
Consider the radial positive solution $v$: $\Delta v+\lambda v=0$ on $\mathbb{D}$ such that $v(z)= v(|z|)$ and $v(0)=1$. Then the ratio $h=\frac{u}{v}$ tends to zero near the boundary of $\mathbb{D}$.

 The zero set and nodal domains of $h$ and $u$ are the same.
We claim that each nodal domain of $h$ contains at least one of the points $p_k$.

 First, we note that the function $u$ (and h) tend to $+\infty$ near $p_k$ if $c_k<0$ and to $-\infty$ if $c_k>0$. This can be seen by considering the Mobius transformation $T$ such that $\Delta g(T(z))+\lambda g(T(z))= -\delta_{p_k}$. Then by elliptic regularity $u(z)+c_kg(T(z))$ is smooth near $p_k$ and therefore $u$ tends to $\pm\infty$ near $p_k$ depending on the sign of $c_k$.

If a nodal domain $\Omega$ of $h$ did not contain any of $p_k$, then 
 $|h|$ must have a local maximum in $\Omega$ because $h(z) \to 0$ as $z\to \partial \Omega \cup \partial \mathbb{D}$ in Euclidean metric. We claim that $|h|$ cannot have local maxina.

Assume the contrary and identify Poincare's disc with Lobachevsky upper half-plane $\mathbb{H}$.   In the $(x,y)$ coordinates on $\mathbb{H}=\mathbb{R}_{+}^{2}$ 
the functions $u$ and $v$ satisfy $$y^{2}\Delta v+\lambda v=0, \quad  y^{2}\Delta u+\lambda u=0$$
in $\Omega$, which yields that $h=u/v$ satisfies
\[
\text{div}(v^{2}\nabla h)=0 \quad \text{in } \Omega.
\]

By the maximum principle for elliptic equations in divergence form,
the function $h$ cannot have local maxima and minima unless it is identically constants. Contradiction is obtained and therefore each nodal domain of $u$ must contain at least one of the points $p_k$ and therefore the number of nodal domains is at most $n$.

\end{proof}

The following question is relevant to extending Theorem \ref{Main 2} to the  case $\alpha>-1/2$, which we currently don't know how to do.

\textbf{Question on the number of nodal domains for solutions to $\Delta u+\lambda u =0$}. 
Consider the radial solution to $\Delta_\mathbb{D} u-\beta(\beta+1)u = 0$ in $\mathbb{D}$ such that $u(0)=1$. Let $c_1, \dots c_n$ be non-zero real numbers and let $T_1,T_2,\dots, T_n$ be Mobius transformations of $\mathbb{D}$ such that $T_k^{-1}(0)$ are distinct points. Is it true that $v=\sum_{k=1}^{n} c_ku(T_k(z))$ has at most $Cn$ nodal domains?

Despite not being able to answer this question, we now show how to prove Theorem \ref{Main 2}  using 
Lemma \ref{Lemma:main}.
The Cayley transform
\[
\Phi(z) = \frac{z-i}{z+i}
\]
 is an isometry from $\mathbb{H}$ to $\mathbb{D}$, which sends $i$ to $0$ and satisfies
\[
1 - \bigl|\Phi(x+iy)\bigr|^2 = \frac{4y}{x^2 + (y+1)^2}.
\]
Then $g_\mathbb{H}=g(\Phi(z))$ is a fundamental solution to 
$\Delta_\mathbb{H}u+\lambda u = -\delta_{i}$, which is radial with respect to the point $i$ in the hyperbolic metric.
Using
\[
g(z) = C_2(\beta)\,\bigl(1 - |z|^2\bigr)^{\beta+1}\bigl(1+o(1)\bigr),
\quad |z|\to 1-,
\]
we obtain
\begin{equation} \label{eq: asymptotics on half-plane}
g_\mathbb{H}(x+iy)
= C_2(\beta)\left(\frac{4y}{x^2 + (y+1)^2}\right)^{\beta+1}\bigl(1+o(1)\bigr),
\qquad y\to +0.
\end{equation}
Similarly, the function $u_{\mathbb{H}}=u(\Phi(z))$
satisfies 
\begin{equation} 
u(x+iy)
= C_1(\beta)\left(\frac{4y}{x^2 + (y+1)^2}\right)^{-\beta}\bigl(1+o(1)\bigr),
\qquad y\to +0.
\end{equation}

We stress that the term $o(1)$ is uniform in $x$. For the sake of simplicity we will identify $g_\mathbb{H}$ and $g$ (respectively $u_\mathbb{H}$ and $u$) and sometimes will skip writing the index $\mathbb{H}$.

We can find Mobius transformations $T_k:\mathbb{H} \to \mathbb{H}$ sending $a_k+ib_k$ to $i$ (see Appendix for the definition of Mobius transforms) and
a linear combination $$v=\sum _{k=1}^{n}d_k g_\mathbb{H}(T_k(z))$$
 such that for $\alpha=-\beta-1$ our function
\[
f(x)=\sum_{k=1}^{n}c_{k}[(x-a_{k})^{2}+b_{k}^{2}]^{\alpha}= \sum_{k=1}^{n}c_{k}[(x-a_{k})^{2}+b_{k}^{2}]^{-\beta-1}
\]
is a normalized trace of $v$ in the sense that 
$$ v(x+iy)= y^{\beta+1}(f(x)+o(1))$$
as $y\to+0$. 

By Lemma~\ref{Lemma:main} the function $v$ has at most $n$ nodal domains.
By a limit point of the zero set of $v$ we will mean a point $\zeta \in \partial \mathbb{H}$ such that there is a sequence $z_k$ of zeroes of $v$ in $\mathbb{H}$ such that
$z_k \to \zeta$ in Euclidean metric (one can also add $\infty$ to $\partial \mathbb{H}$ and say that $\infty$ is the limit point of zero set of $v$ if there is a sequence of zeroes $z_k$  of $v$ with $|z_k|\to +\infty$).

  We assume that $f$ is not identically zero and by real-analyticity the zero set of $f$ on the $x$-axis is locally finite.
We will show that the zero set of $v$ in $\mathbb{H}$ has at most $2(n-1)$ limit points on $\partial \mathbb{H}$ and every zero of $f$ on $\mathbb{R}$ is a limit point, which will finish the proof.

We repeat the topological argument from Section \ref{Sec: half integeres}. In view of Bers' equiangular theorem, see Figure \ref{fig:perturbation},
we can infinitesimally perturb the zero set of $v$ in $\mathbb{H}$ near the singular points in such a way that the set becomes a disjoint union of smooth curves in $\mathbb{H}$, the number of nodal domains does not increase and the set of limit points remains the same. Hence the infinitesimal perturbation of the zero set of $v$ has no more than $(n-1)$ smooth curves and therefore the zero set of $v$ in $\mathbb{H}$ is a union of at most $n-1$ piece-wise smooth curves, which intersect only at singular points. 
 The set of limit points of the zero set of $v$ is discrete because it is a subset of the zeroes of real analytic $f$ and $ v(x+iy)= y^{\beta+1}(f(x)+o(1))$. By real analyticity every zero curve of $v$ can't have more than $2$ limit points on the line $\{(x,0)\} $, and we conclude that the total number of limit points is at most $2(n-1)$.

It is now sufficient to show that every zero of $f$ is a limit point of the zero set of $v$. Clearly, if $f$ changes sign at some point $p \in \mathbb{R}$, then $v$ also changes sign in any neighborhood of $p$ in the upper half-plane, and thus $p$ is a limit point. Recall that $f$ is real-analytic. Arguing by contradiction in the case when $f$ does not change sign at one of its zeros, we may assume without loss of generality that
\[
f(0) = 0,\qquad f(x) > 0 \quad \text{for } 0<|x|<\varepsilon,
\]
and also that $v(z)$ is a (strictly) positive solution to $\Delta v + \lambda v = 0$ in 
\(
\mathbb{H} \cap \{ |z|\le \varepsilon \}.
\)

We want to arrive at a contradiction with these assumptions. First note that, by the positivity of $v$ and the asymptotics of $v$ near the real line, we have
\[
v(x+iy) \ge c\, y^{\beta+1} \quad \text{for } |x+iy| \in \Bigl[\frac{\varepsilon}{4},\frac{3\varepsilon}{4}\Bigr]
\]
for some $c>0$. Let us note that $y^{-\beta}$ and $y^{\beta+1}$ are both positive solutions to
\[
\Delta_{\mathbb{H}} u - \beta(\beta+1)u = 0.
\]
Consider the ratio of two solutions to the same PDE:
\[
h = \frac{v - c y^{\beta+1}}{y^{-\beta}},
\]
which is a solution to
\[
\operatorname{div}(y^{-2\beta} \nabla h)=0
\quad \text{in } \mathbb{H}\cap\{|z|\le \varepsilon\}.
\]
Therefore $h$ cannot have local minima unless it is a constant function.  We recall that 
$$ v(x+iy)= y^{\beta+1}(f(x)+o(1))$$
as $y\to+0$ and therefore
\[
|h(x+iy)| \le C\, y^{2\beta+1}
\quad \text{in } \mathbb{H}\cap\Bigl\{|x+iy|\le \frac{3\varepsilon}{4}\Bigr\},
\]
which yields $h(x+iy) \to 0 \quad \text{as } y\to 0.$
Also, $h \geq 0$ for 
\[
|x+iy|\in \Bigl[\frac{\varepsilon}{4},\frac{3\varepsilon}{4}\Bigr].
\]
By the minimum principle this implies that $h\ge 0$ in
\(
\mathbb{H}\cap\Bigl\{|x+iy|\le \frac{3\varepsilon}{4}\Bigr\},
\)
and therefore
\[
v(x+iy) \ge c\, y^{\beta+1}
\quad \text{in } \mathbb{H}\cap\Bigl\{|x+iy|\le \frac{3\varepsilon}{4}\Bigr\}.
\]
This contradicts the asymptotics of $v$ at $0$, since $f(0)=0$.

\section{Appendix}

\subsection*{Axially symmetric harmonic functions in Euclidean space and Laplace-Beltrami eigenfunctions on Lobachevsky upper half--plane}

This subsection is not formally needed for the proof of Theorem \ref{Main 2}. 
The purpose of this section is rather
to explain a connection between axially symmetric harmonic functions and  solutions to
\(\Delta u + \lambda u =0 \) on the Lobachevsky plane. This connection creates a link between two different proofs and, as byproduct, yields a curious transformation of axially symmetric harmonic functions in
\[
\mathbb{R}^3 \setminus \{(x,0,0) : x \in \mathbb{R}\}
\]
to solutions of \(\Delta u + \frac{1}{4}u = 0\) on the Lobachevsky hyperbolic plane.

Namely, an axially symmetric function
\[
u(x_1,x_2,x_3)=h\bigl(x_1, \sqrt{x_2^2+x_3^2}\bigr)
\]
in \(\mathbb{R}^3 \setminus \{(x_1,0,0): x_1\in\mathbb{R}\}\) is harmonic outside the \(x_1\)-axis if and only if
\[
\varphi(x,y)=h(x,y)\,y^{1/2}
\]
is a solution to \(\Delta_\mathbb{H}\varphi+\frac{1}{4}\varphi=0\) on the Lobachevsky upper half-plane \(\mathbb{H}\), where
the Laplace--Beltrami operator on \(\mathbb{H}\) is defined by
\[
\Delta_\mathbb{H}u = y^2(u_{xx}+u_{yy}).
\]

Once the statement is formulated, it is straightforward to verify:
\[
\Delta_\mathbb{H}\varphi+\frac{1}{4}\varphi=0 
\iff \Delta_\mathbb{H}(h y^{1/2})+\tfrac{1}{4}h y^{1/2}=0 
\iff h_{xx}+h_{yy}+\frac{h_y}{y}=0 \iff \Delta u=0.
\]

We also remark that \(y^{1/2}\) itself solves the equation \(\Delta_\mathbb{H}\varphi+\frac{1}{4}\varphi=0\). Thus \(h\) can be viewed as the ratio of two solutions of the same eigenvalue equation, which explains the extra factor \(y^{1/2}\) in the transformation from eigenfunctions with eigenvalue $1/4$ to harmonic functions.

 Similarly, if we consider the $x_1$-axis $\{(x_1,0,\dots,0)\}$ in $\mathbb{R}^n$ and a function $u(x_1,\dots,x_n)=h(x_1,\sqrt{x_2^2+\dots+x_n^2})$, which is invariant with respect to rotations preserving the $x_1$-axis, then $u$ is harmonic in 
 $\mathbb{R}^n\setminus \{(x_1,0,\dots,0)\}$ if and only if $\varphi=h(x,y)y^{\frac{n-2}{2}}$ satisfies 
 $$\Delta_\mathbb{H} \varphi - \frac{(n-2)(n-4)}{4}\varphi=0.$$

\subsection*{Radial Fundamental solution to $\Delta + \frac{1}{4} $ on $\mathbb{H}$ and Newtonian potential of a uniform charge on a circular wire.}
We would like to make a remark that is not used in the proofs, but which
motivates studying fundamental solutions on the Lobachevsky plane in
order to generalize the proof of Theorem \ref{Main 2} from the
half-integer case.

Let $\mu$ be the length measure on the circle
$S=\{(0,\cos \theta, \sin \theta)\} \subset \mathbb{R}^3$. Then the
Newtonian potential
\[
u := \mu * \frac{1}{|x|}
\]
is an axially symmetric harmonic function on $\mathbb{R}^3 \setminus S$,
which tends to $+\infty$ near the wire $S$ and satisfies
$\Delta u = -4\pi \mu$ in the sense of distributions. By rotational symmetry,
we can rewrite $u$ in terms of
\[
h(x_1,\sqrt{x_2^2+x_3^2}) = u(x_1,x_2,x_3)
\quad\text{and}\quad
\varphi(x,y) = h(x,y)\,y^{1/2},
\]
so that $\Delta u = -4\pi \mu$ becomes
\[
\Delta_\mathbb{H} \varphi + \frac{1}{4} \varphi = -c \,\delta_{(0,1)}
\]
for some numerical constant $c>0$.

We claim that $\varphi$ depends only on the hyperbolic distance to
$(0,1)$; more precisely,
\[
\frac{x^2+y^2+1}{2y}= \cosh d_\mathbb{H}\big((x,y),(0,1)\big).
\]
Indeed,
\begin{equation} \label{eq:radial}
\begin{aligned}
\varphi(x,y)
  &:= h(x,y)\,y^{1/2}
   = y^{1/2}\int_0^{2\pi}
       \frac{d\theta}{\sqrt{x^2 + (y-\cos\theta)^2 + \sin^2\theta}} \\
  &= \int_0^{2\pi}
       \frac{d\theta}{\sqrt{\frac{x^2 + y^2 + 1}{y} - 2\cos\theta}}.
\end{aligned}
\end{equation}

Thus $\varphi(x,y)$ is the radial fundamental solution of
\[
\Delta_\mathbb{H} \varphi + \frac{1}{4} \varphi = -c \,\delta_{(0,1)}
\]
for some numerical $c>0$, and it has the asymptotic behavior
\[
\varphi(x,y)= y^{1/2}\frac{2\pi}{\sqrt{x^2+1}}(1+o(1))
\]
as $y\to 0^+$. This suggests extending
the proof of Theorem \ref{Main 2} to the case $\alpha < -1/2$ by
considering radial fundamental solutions to $\Delta +\lambda$ on
$\mathbb{H}$ with $\lambda<1/4$, instead of axially symmetric
potentials in $\mathbb{R}^n$.

\subsection*{Isometries of Lobachevsky upper half-plane}
We will identify Lobachevsky upper half-plane $\Hh$ with complex numbers $z=x+iy$ with $y>0$. Poincare's metric $ds^2=\frac{dx^2+dy^2}{y^2}$ on $\Hh$ is invariant with respect to 
Mobius transformations:
$$z\to \frac{az+b}{cz+d},$$
where $a,b,c,d \in \mathbb{R},\ ad - bc = 1$. It is convenient to identify the group of $2\times2$ matrices with determinant $1$
\[
PSL(2,\mathbb{R}):=
\left\{
\begin{pmatrix}
a & b\\
c & d
\end{pmatrix}
: a,b,c,d \in \mathbb{R},\ ad - bc = 1
\right\}/ \{\pm I\}
\]
 with the group of Mobius transformations. For $A=\begin{pmatrix}
a & b\\
c & d
\end{pmatrix}$ and a complex number $z$ in $\Hh$, we will denote $\frac{az+b}{cz+d}$ by $A\cdot z$. 
Like Poincar\'e's metric, the Laplace--Beltrami operator on the upper half--plane
\[
\Delta_{\mathbb H} u = y^{2}\bigl(u_{xx} + u_{yy}\bigr)
\]
is invariant under M\"obius transformations. In particular, if a function
\(\varphi\) satisfies
\[
\Delta_{\mathbb H} \varphi + \lambda\varphi = 0,
\]
then, for every M\"obius transformation \(A\) of \(\mathbb H\), the function
\(\varphi(A\cdot z)\) is also an eigenfunction with the same eigenvalue.
\subsection*{Discrete groups generated by the Kelvin transform and similarities in $\mathbb{R}^3$ and modular harmonic functions.}
Suppose we are given an eigenfunction \(\varphi\) of the hyperbolic Laplacian,
\[
\Delta_{\mathbb H} \varphi + \frac{1}{4}\varphi = 0,
\]
which is invariant under a discrete subgroup of \(SL(2,\mathbb R)\) acting by
isometries on \(\mathbb H\). A natural question is how these symmetries of
\(\varphi\) are transferred to the corresponding axially symmetric harmonic
function.
\begin{itemize}
    \item Translation $(x,y) \to (x+t, y)$ is an isometry on $\Hh$ and it  corresponds to a translation $(x_1, x_2, x_3) \to (x_1+t, x_2, x_3)$ in $\R^3$, which preserves harmonicity.
    \item Homothety $(x,y) \to (tx, ty)$ is also an isometry on $\Hh$ and 
     it  corresponds to homothety $(x_1, x_2, x_3) \to (tx_1, tx_2, tx_3)$ in $\R^3$.
    \item The inversion $z \to {-1/z}$, corresponds to the Kelvin transform in $\R^3$, given by $u(x) \to |x|^{-1} u(x/|x|^2)$, which preserves harmonicity in $\mathbb{R}^3\setminus\{0\}$.
\end{itemize}
 If an eigenfunction $\varphi$, $\Delta_{\mathbb{H}}\varphi+\frac{1}{4}\varphi=0$ has a symmetry, say $\varphi(z)=\varphi\left(\frac{az+b}{cz+d}\right)$ for some $\begin{pmatrix}
a & b\\
c & d
\end{pmatrix}\in SL(2, \R)$, then axially symmetric harmonic function $u$ in $\mathbb{R}^3$, which we will identify with $h(z):=\varphi(z)\Im(z)^{-1/2}$, behaves like a modular form of weight $1$:
\begin{equation}
    h\left( \frac{az+b}{cz+d}\right)= |cz+d|h(z).
\end{equation}
It follows from the calculation of imaginary part
$$  \Im\!\left(\frac{az + b}{cz + d}\right)=\Im\!\left(\frac{(az + b)\,\overline{(cz + d)}}{|cz + d|^2}\right)= \frac{\Im(z)(ad-bc)}{|cz + d|^2}=\frac{\Im(z)}{|cz + d|^2}.$$
\subsection*{One dimensional modular distributions of weight $5$ and bounded eigenfunctions with eigenvalue $\frac14$ on $\mathbb{H}$}

Suppose we have a bounded solution $\varphi$ to
\[
\Delta_{\mathbb H}\varphi + \frac{1}{4}\varphi = 0
\]
in the upper half–plane. As before, we associate to $\varphi$ an axially symmetric harmonic function
\[
u(x_1,x_2,x_3) = \varphi(x_1,r)\,r^{-1/2}, \qquad r = \sqrt{x_2^2 + x_3^2},
\]
which is harmonic in $\R^3\setminus\{(x_1,0,0):x_1\in\R\}$. Since $\varphi$ is bounded and
$r^{-1/2}$ is locally integrable with respect to the measure $r\,dr\,d\theta\,dx_1$ in cylindrical coordinates, 
$u$ is locally integrable in $\R^3$, and we can consider the distribution
\[
\nu := \Delta u \in \mathcal D'(\R^3).
\]

By construction, $u$ is harmonic away from the $x_1$–axis, hence
\[
\text{supp} (\nu) \subset \{(x_1,0,0):x_1\in\R\}.
\]
One can informally think that $u$ is a Newtonian potential generated by a distribution $\nu$ supported on $x_1$ axis.

\textbf{Claim} If a harmonic function $u$ in $\R^3\setminus\{(x_1,0,0):x_1\in\R\}$ satisfies $|u|\leq C r^{-1/2}$, then for any $f \in C_c^{\infty}(\mathbb{R}^3)$ vanishing on the $x_1$-axis, we have
$$ <\Delta u, f>:=\int_{\mathbb{R}^3} u \Delta f=0.$$

We postpone the proof of the claim for a few sentences. The claim implies that there is a one-dimensional distribution $\mu \in \mathcal{D}'(\R)$ such that
$$ <\nu,f>=<\mu(t),f(t,0,0)>.$$
In particular, the claim implies that $\mu=\Delta u$ does not contain derivatives of measures in the radial directions like $\delta_{x_2x_2}+\delta_{x_3x_3}$. 
We will call $\mu$ the restriction of $\nu$ to the $x_1$-axis.

\begin{proof}

Let $r=\sqrt{x_2^2+x_3^2}$ be the distance to the $x_1$-axis. By assumption,
$u$ is harmonic in $\R^3\setminus\{r=0\}$ and
\[
|u(x)| \;\leq\; C\,r(x)^{-1/2}.
\]
By Cauchy estimates, it yields a bound for the gradient of $u$: 
 \[
|\nabla u(x)| \;\leq\; C_1 \frac{\max
\limits_{B_{r(x)/2}(x)}|u|}{r(x)/2} \leq C_2 \,r(x)^{-3/2}.
\]

Since $f\in C_c^\infty(\R^3)$ vanishes on the $x_1$-axis, we have
\begin{equation}\label{eq:f-r}
|f(x)| \le C_f\,r(x),
\qquad
|\nabla f(x)| \le C_f,
\end{equation}
for some constant $C_f$ depending on $f$.

Let $R>0$ be such that $\text{supp} f\subset B_R(0)$. For $0<\varepsilon<R$ set
\[
\Omega_\varepsilon := B_R(0)\cap\{r>\varepsilon\}.
\]
Then $u$ is harmonic in $\Omega_\varepsilon$, so by Green’s formula
\begin{equation}\label{eq:ibp}
\int_{\Omega_\varepsilon} u\,\Delta f\,dx=\int_{\Omega_\varepsilon} (u\,\Delta f - f\,\Delta u)\,dx
= \int_{\partial\Omega_\varepsilon} (u\,\partial_n f - f\,\partial_n u)\,dS 
= \int_{\Sigma_\varepsilon} (u\,\partial_n f - f\,\partial_n u)\,dS,
\end{equation}
where
\[
\Sigma_\varepsilon := \{x\in B_R(0): r(x)=\varepsilon\}
\]
is the inner cylindrical part of  $\partial\Omega_\varepsilon$. We note that the area of 
$\Sigma_\varepsilon$ is $O(\varepsilon)$ as $\varepsilon \to 0$.

Using the Cauchy estimates for $u$, we get 

$$\int_{\Sigma_\varepsilon} (u\,\partial_n f - f\,\partial_n u)\,dS= O(\varepsilon^{1/2}).$$

Passing to the limit as $\varepsilon\to 0$ in \eqref{eq:ibp} we obtain
$
\int_{\R^3} u\,\Delta f\,dx = 0.
$

\end{proof}

One can ask what if $\varphi$ is invariant under $A = \begin{pmatrix}a & b\\ c & d\end{pmatrix}\in SL(2,\R)$, i.e.
\[
\varphi\left(\frac{az+b}{cz+d}\right) = \varphi(z)
\quad\text{for all }z\in\Hh,
\]
and how it translates to symmetries of $\mu$.
The way $\mu$ transforms under hyperbolic isometries is dictated by the behavior of the Laplacian under the Kelvin transform in $\R^3$,
\[
(Ku)(x) = |x|^{-1}u\!\left(\frac{x}{|x|^2}\right).
\]
For $u\in C^\infty(\R^3\setminus\{0\})$ one has
\begin{equation}\label{eq:Kelvin-laplacian}
\Delta (Ku)(x)
= |x|^{-5}\,\bigl(\Delta u\bigr)\!\left(\frac{x}{|x|^2}\right)
\quad\text{in }\R^3\setminus\{0\}.
\end{equation}
If $\varphi$ is invariant under the inversion, i.e.
\[
\varphi\!\left(-\frac{1}{z}\right) = \varphi(z),
\]
then the corresponding axially symmetric harmonic function $u$ is invariant under the Kelvin transform, $Ku=u$. 
Combining this with \eqref{eq:Kelvin-laplacian} it formally (but not mathematically) gives
\[
\nu(x) = |x|^{-5}\,\nu\!\left(\frac{x}{|x|^2}\right),
\]
but the right-hand side does not  make sense in terms of distributions in $D'(\R^3)$. We can try to formulate the modular property in terms of test functions as
\begin{equation}\label{eq:weight5-inversion}
  \big\langle \mu,f\big\rangle
  = \left\langle \mu,\;|t|^{-5}\,
      f\!\left(-\frac{1}{t}\right)\right\rangle.
\end{equation}
and the latter identity is straightforward to verify for $f\in C_c^\infty(\mathbb{R}\setminus\{0\})$ (using that $Ku=u$).
However if $f\in C_c^\infty(\mathbb{R})$, then the function $|t|^{-5}\, f\!\left(-\frac{1}{t}\right)$ can be smoothly extended by $0$ at the origin, but it is not compactly supported, so the right hand side is not formally defined and one has to work with a different space of test functions
$$Q_5(\mathbb{R}):=
\{f\in C^\infty(\mathbb{R}) \text{ such that } |t|^{-5}\, f\!\left(\frac{1}{t}\right) \text{ has a } C^\infty( \mathbb R) \text{ extension at } 0\}$$
with convergence $f_n \to f$ in $Q_5$ defined by two conditions: $f_n \to f$ in $C^\infty$ and $|t|^{-5}\,f_n\!\left(\frac{1}{t}\right) \to |t|^{-5}\,f\!\left(\frac{1}{t}\right) $ in $C^\infty$.
Every function $f$ in $Q_5$ can be written (non-uniquely) as a sum $$f=f_1+|t|^{-5}\,f_2\!\left(\frac{1}{t}\right)$$ 
for some $f_1,f_2 \in C_c^\infty(\mathbb{R})$.
It is straightforward to check that the space $Q_5$ is invariant with respect to transformations $$f(t) \to |ct+d|^{-5}\,f\!\left(\frac{at+b}{ct+d}\right)$$
for $A = \begin{pmatrix}a & b\\ c & d\end{pmatrix}\in SL(2,\R)$.

If $\varphi$ is invariant under $A = \begin{pmatrix}a & b\\ c & d\end{pmatrix}\in SL(2,\R)$, and $c\neq0$,
then the constructed one–dimensional distribution $\mu$ can be extended as continuous linear functional from $C_c^\infty(\mathbb{R})$ to $Q_5$ and it should satisfy
\begin{equation}\label{eq:weight5-general}
  \big\langle \mu,f\big\rangle
  = \left\langle \mu,\;
     |ct+d|^{-5}\,
     f\!\left(\frac{at+b}{ct+d}\right)\right\rangle,
  \qquad f\in Q_5(\R). 
\end{equation}
We skip the non-difficult proof of \eqref{eq:weight5-general}. This observation was not used in this article. 

\subsection*{Asymptotics of radial solutions to $\Delta_\mathbb{D} u-\beta(\beta+1) u =0$ in $\mathbb{D}\setminus\{0\}$.}

In this section we prove Lemma \ref{Lemma:asymptotics}.
We begin with a simple observation that the functions
\[
u_1(z) = y^{-\beta}, \qquad u_2(z) = y^{\beta+1}, \qquad z=x+iy,
\]
both satisfy 
\[
\Delta_{\mathbb{H}} u - \beta(\beta+1)u = 0 \quad \text{in } \mathbb{H},
\]
where
\[
\Delta_{\mathbb{H}} = y^{2}\left(\frac{\partial^{2}}{\partial x^{2}}
+ \frac{\partial^{2}}{\partial y^{2}}\right).
\]

Next we pass to the Poincaré disk model
\[
\mathbb{D} = \{z\in\mathbb{C} : |z|<1\}
\]
via conformal mapping
\[
\psi:\mathbb{D}\to\mathbb{H}, \qquad
\psi(z) = i\,\frac{1+z}{1-z}.
\]
This map is an isometry between the hyperbolic metrics on $\mathbb{D}$ and $\mathbb{H}$, and if $u := v\circ\psi$, then
\[
\Delta_{\mathbb{H}} v - \beta(\beta+1)v = 0 \quad\text{in }\mathbb{H},
\]
is equivalent to
\[
\Delta_{\mathbb{D}} u - \beta(\beta+1)u = 0 \quad\text{in }\mathbb{D}.
\]

A standard computation shows that
\[
\Im\psi(z)
= \frac{1}{2i}\bigl(\psi(z)-\overline{\psi(z)}\bigr)
= \frac{1}{2i}\left(i\frac{1+z}{1-z} + i\frac{1+\bar z}{1-\bar z}\right)
= \frac{1-|z|^2}{|1-z|^2}, \qquad z\in\mathbb{D}.
\]

Therefore, the images of the two elementary eigenfunctions $y^{-\beta}$ and
$y^{\beta+1}$ under the conformal mapping are
\[
u_1(z) = (\Im\psi(z))^{-\beta}
= \left(\frac{1-|z|^2}{|1-z|^2}\right)^{-\beta}
= \left(\frac{|1-z|^2}{1-|z|^2}\right)^{\beta},
\]
and
\[
u_2(z) = (\Im\psi(z))^{\beta+1}
= \left(\frac{1-|z|^2}{|1-z|^2}\right)^{\beta+1},
\]
and both satisfy
\[
\Delta_{\mathbb{D}} u - \beta(\beta+1)u = 0 \quad\text{in }\mathbb{D}.
\]

Rotations of the disc
\[
R_\zeta(z) = \bar\zeta\, z, \qquad \zeta\in\partial\mathbb{D},
\]
also preserve the hyperbolic metric and hence the operator $\Delta_{\mathbb{D}}$, so they send eigenfunctions of $\Delta_{\mathbb{D}}$ to eigenfunctions with the same eigenvalue. In particular,
\[
K_\zeta(z)
:= u_1(R_\zeta z)
= \left(\frac{|R_\zeta z-1|^2}{1-|R_\zeta z|^2}\right)^{\beta}
= \left(\frac{|z-\zeta|^2}{1-|z|^2}\right)^{\beta}, \qquad \zeta\in\partial\mathbb{D},
\]
solves
\[
\Delta_{\mathbb{D}} K_\zeta - \beta(\beta+1)K_\zeta = 0 \quad\text{in }\mathbb{D}.
\]

By linearity of the PDE, rotational symmetrisation of these solutions,
\begin{equation} \label{eq:radial 2}
u_\beta(z)
:= \frac{1}{2\pi}\int_{\partial\mathbb{D}} K_\zeta(z)\,d\mu(\zeta)
= \frac{1}{2\pi}\int_{\partial\mathbb{D}}
\left(\frac{|z-\zeta|^2}{1-|z|^2}\right)^{\beta} d\mu(\zeta),
\end{equation}
is again a solution of
\[
\Delta_{\mathbb{D}} u - \beta(\beta+1)u = 0 \quad\text{in }\mathbb{D},
\]
and by construction it is radial: $u_\beta(z)=u_\beta(|z|)$ and by \eqref{eq:radial 2} we have $u_\beta(0)=1$. We claim that such a radial solution is unique. In particular, this implies that
\[
u_\beta(z) = u_{-1-\beta}(z).
\]

Suppose, by contradiction, that there exist two linearly independent radial
solutions of $\Delta_{\mathbb{D}} u - \beta(\beta+1)u = 0$ in $\mathbb{D}$.
Then we can take a non-zero linear combination $u$ such that $u(0)=0$. Let $k$
be the vanishing order of $u$ at $0$.

By Bers' equiangular theorem, the zero set of $u$ near $0$ is a union of 
$C^1$-smooth curves intersecting at $0$ with equal angles between consecutive curves.
On the other hand, since $u$ is radial, its zero set near $0$ consists of circles centered at the origin and $0$. This is a
contradiction. Hence there is at most one radial solution, and the uniqueness
follows.

From now on we will omit $\beta$ in $u_\beta$ and we will simply write $u$. 

Assume now that $\beta>-1/2$, which was not used before this point. In this
case
\[
C_1(\beta)=\frac{1}{2\pi}\int_{\partial\mathbb{D}}
|1-\zeta|^{2\beta} \, d\mu(\zeta)
\]
is a finite positive numerical constant, and it follows directly from
\eqref{eq:radial 2} (by dominated convergence) that
\[
u(z)
= C_1(\beta)\,(1 - |z|^2)^{-\beta}\bigl(1+o(1)\bigr)
\quad\text{as }|z|\to 1^-.
\]
This asymptotic formula implies that
\begin{equation} \label{eq: radial fundamental 2}
g(z)= \frac{u(|z|)}{2\pi}\int_{|z|}^{1}\frac{1}{u(r)^2\,r}\,dr
\end{equation}
is a well-defined radial function in $\mathbb{D}\setminus\{0\}$. 

For a radial function $g$ we have
\[
\Delta_{\mathbb{D}} g(r)
= \frac{(1-r^2)^2}{4}\left(g''(r)+\frac{1}{r}g'(r)\right).
\]
Let us show
that
\[
\Delta_{\mathbb{D}} g - \beta(\beta+1)g = 0 \quad\text{in }\mathbb{D}\setminus\{0\},
\]
which is equivalent to
\[
\frac{(1-r^2)^2}{4}\left(g''(r)+\frac{1}{r}g'(r)\right)
- \beta(\beta+1)g(r)=0,\qquad 0<r<1.
\]

By construction, $u(r)$ is a positive radial solution of the same ODE. Write
\[
g(r)=u(r)\,h(r)
\]
and substitute into the equation. Using that $u$ itself satisfies this ODE, the
terms involving $h$ cancel and we obtain the reduced equation
\[
(1-r^2)^2\left(h''(r)+\Bigl(\frac{1}{r}+2\frac{u'(r)}{u(r)}\Bigr)h'(r)\right)=0,
\]
which is a first-order linear ODE for $h'$, which can be rewritten as
\[
-\left[\log(ru^2) \right]'=-\left(\frac{1}{r}+2\frac{u'}{u}\right)
=\frac{h''}{h'}=\left[\log |h'|\right]'.
\]
Solving it gives
\[
h'(r)=\frac{C}{r\,u(r)^2},\qquad 0<r<1,
\]
for some constant $C$. We can choose $C=\frac{-1}{2\pi}$ and put $h(1)=0$. Integrating from $r$ to $1$ we obtain
\[
h(r)=\frac{1}{2\pi} \int_r^1\frac{1}{u(t)^2\,t}\,dt.
\]
This choice yields a solution
\[
g(r)=u(r)h(r)=\frac{u(r)}{2\pi}\int_r^1\frac{1}{u(t)^2\,t}\,dt.
\]

Next, we investigate the asymptotic behaviour of $g$ at the boundary and at the
origin. Using the asymptotics of $u$ as $r\to1^-$ and
\eqref{eq: radial fundamental 2}, one checks directly that
\[
g(z)
= C_2(\beta)\,(1 - |z|^2)^{\beta+1}\bigl(1+o(1)\bigr), \qquad C_2(\beta) >0,
\]
as $|z|\to 1^-$. On the other hand, since $u(0)=1$ and $u$ is smooth near $0$,
we have $u(r)=1+O(r^2)$ as $r\to0$, so
\[
\int_r^1\frac{1}{u(t)^2\,t}\,dt
= \int_r^1\frac{1}{t}\,dt + O(1)
= -\log r + O(1),
\]
 hence
\[
g(z)\sim \frac{-1}{2\pi}\log|z|\quad\text{and similarly }\quad
\frac{\partial g}{\partial r}(z)\sim \frac{\partial}{\partial r}
\left(\frac{-1}{2\pi}\log|z|\right)
\]
as $z\to 0$. This is consistent with the expected behavior of a fundamental
solution. It also shows that $g$ is locally integrable in $\mathbb{D}$ (the
logarithmic singularity is integrable).

Finally, let $\lambda:=-\beta(\beta+1)$.
By saying that $\Delta_{\mathbb{D}} g + \lambda g = -\delta_0$ (or that $g$ is a fundamental solution of the Laplace--Beltrami operator plus $\lambda$ up to a sign) we mean that $g$ is locally integrable and that for any $\varphi\in C_c^\infty(\mathbb{D})$,
\begin{equation} \label{eq:distributions}
\int_{\mathbb{D}} g\,(\Delta_{\mathbb{D}}\varphi+\lambda\varphi)\,dA_{\mathbb{D}}
= -\varphi(0),
\end{equation}
where $dA_{\mathbb{D}}$ denotes the hyperbolic area measure.
To prove that $g$ is a fundamental solution, we follow the classical argument showing that
$\frac{1}{2\pi}\log|z|$ is a fundamental solution for the Euclidean Laplacian.

To see this, fix
$\varepsilon>0$ small and integrate over $\mathbb{D}\setminus B_\varepsilon$,
where $B_\varepsilon=\{z:|z|<\varepsilon\}$. Since
$\Delta_{\mathbb{D}} g + \lambda g = 0$ on $\mathbb{D}\setminus\{0\}$, Green’s
formula gives
\[
\int_{\mathbb{D}\setminus B_\varepsilon} g\,(\Delta_{\mathbb{D}}\varphi+\lambda\varphi)\,dA_{\mathbb{D}}
= \int_{\partial B_\varepsilon}\bigl(g\,\partial_n\varphi-\varphi\,\partial_n g\bigr)\,dS_{\mathbb{D}},
\]
where $dS_{\mathbb{D}}$ is the
hyperbolic length and $\partial_n$ is the outward normal derivative. The outer boundary term vanishes because $\varphi$ has
compact support in $\mathbb{D}$.

Using the asymptotics of $g$ and $\partial_n g$ near $0$, and the Taylor
expansion $\varphi(z)=\varphi(0)+O(|z|)$, the right-hand side converges, as
$\varepsilon\to0$, to $-\varphi(0)$. Passing to the limit
$\varepsilon\to0$ yields \eqref{eq:distributions}, which means that $\Delta_{\mathbb{D}}g+\lambda g=-\delta_0$ in the sense of
distributions. This completes the proof of the lemma.

\subsection*{Number of nodal domains for solutions to $\Delta u+\lambda u=0$ in a punctured hyperbolic plane  and application to a question of Gabrielov, Novikov and Shapiro.}

In this section we explain how to generalize the proof of Theorem \ref{Main 2} to obtain a linear bound for Conjecture 1.9 from \cite{GNS}.

\begin{theorem} \label{Main 4}
Let $a_k, b_k$ and $c_k$ be real numbers.  Let $\alpha \in (-\infty, -1/2)$. If the function
\[
f(x)=\sum_{k=1}^{n}c_{k}[(x-a_{k})^{2}+b_{k}^{2}]^{\alpha}
\]
is not identically zero, then $f$ has no more than $4n-2$ critical points.
\end{theorem}

We will provide only the sketch of the proof, which is parallel to the proof of Theorem \ref{Main 2} and is based on the following lemma, which generalizes lemma \ref{Lemma:main} to the case of solutions to $\Delta u + \lambda u=0$, $\lambda <1/4$ in Poincare's disk with a finite number of points removed, where the solution does not blow up faster than polynomially.

\begin{lemma} \label{Lemma:main 2}
Let $\beta>-1/2$ and $\lambda=-\beta(\beta+1)< \frac{1}{4}$. Let $u$ be a non-zero solution to
\[
\Delta_{\mathbb{D}} u + \lambda u=0 \quad \text{ in punctured Poincare's disc } \quad \mathbb{D}\setminus \{p_k:k=1, \dots, n\},
\]
where $p_k$ are distinct points in $\mathbb{D}$. Assume that $$u(z) = o(1-|z|^2)^{-\beta} \text{ as } z \to \partial \mathbb{D}$$ and that 
$$|u(z)|= o(|z-p_k|^{-A_k}) \text{ as } z\to p_k$$
for some integers $A_k$.
 Then $u$ has at most $N:=\#\{k:A_k=1\} + \sum2(A_k-1)$ nodal domains in $\mathbb{D}$.
\end{lemma}

We won't give a complete proof of Lemma \ref{Lemma:main 2}, which is based on the argument of Lemma \ref{Lemma:main}. But the required modifications of the argument are outlined below.  First, we review how the solutions of elliptic PDE behave near the singular points $p_k$.

 \textbf{Claim.} Let $u$ be a real solution to $\Delta_{\mathbb D} u +\lambda u=0$ in a punctured neighbourhood of $0$, such that $|u(z)|= O(|z|^{-K})$ as $|z| \to 0$ for some integer $K$, then $u$ has one of the following asymptotics:
\begin{itemize}
    \item There is a $C^\infty$ smooth extension of $u$ through 0 or
    \item $u(z)=(1+o(1))c_0\log|z|$ for some real non-zero $c_0$ or
     \item \begin{equation} \label{eq: Bers pole}
        u(z)=\Re(c_0z^{-k})+o(|z|^{-k}), \quad \nabla u(z)=\nabla\Re( c_0z^{-k})+O(|z|^{-k-1}),
     \end{equation}
as $|z|\to 0$ for some integer $k\leq K$ and $c_0 \in \mathbb{C}\setminus\{0\}$.
\end{itemize}

The claim is known to specialists for more general two-dimensional linear elliptic PDE, but in this case
one can use the symmetry of the space. The idea of the proof is sketched below.
One can decompose $u$ into Fourier series in a punctured neighborhood of zero
$$ u(r,\varphi) = \sum_{k \in \mathbb{Z}}u_k(r,\varphi), $$
where
$$u_k(r,\theta):= \frac{1}{2\pi}\int_0^{2\pi}u(r,\varphi)e^{ik(\theta-\varphi)}\,d\varphi,$$
is a convolution of $u$ with $e^{ik\varphi}$ with respect to rotations around $0$.
In particular, $u_k$ is a solution to the same PDE, $\Delta_{\mathbb D} u_k +\lambda u_k=0$,
in a punctured neighborhood of zero and $u_k(r,\varphi)=v_k(r)e^{ik\varphi}$, where $v_k$ is a solution to linear ODE of second order:
\begin{equation}
v_k''(r) + \frac{1}{r} v_k'(r) - \frac{k^2}{r^2} v_k(r)
    + \frac{4\lambda}{(1-r^2)^2} v_k(r) = 0.
\end{equation}
 
The decomposition $u(r,\varphi) = \sum_{k \in \mathbb{Z}}u_k(r,\varphi)$ is an orthogonal decomposition of $u(r,\cdot)$ into Fourier series on the circle of radius $r$. 
For $k=0$, we have already shown that there are two linearly independent solutions: one has analytic extension at $0$ and the other behaves like $\log|z|$. 

For $|k|>0$ there are two linearly independent solutions, which have asymptotics $r^{-|k|}$ and $r^{|k|}$ at $0$.
Note that  $u(r,\varphi) = O(r^{-K})$ implies $v_k(r)=O(r^{-K})$ as $r \to 0$. For $|k|>K$ the space of the solutions to the ODE for $v_k$ with $v_k(r)=O(r^{-K})$ is one-dimensional (and consists of solutions with asymptotics $cr^{|k|}$ at $0$). In particular, for $|k|>K$ the functions $u_k$ extend smoothly across $0$.

It remains to check that $\sum\limits_{|k|>K} u_k$ converges to a smooth solution in a neighborhood of zero, while the finite sum $\sum\limits_{|k|\leq K} u_k$ has asymptotic behavior of the form \eqref{eq: Bers pole}. We leave out these details.

The claim implies that Bers' equiangular property still holds near point singularities of real solutions to $\Delta u + \lambda u=0$, where $u$ blows up with at most polynomial speed. 

We are ready to finish the proof of Lemma \ref{Lemma:main 2}. Consider the nodal domains $\Omega_i$ of $u$, by which we mean the maximal open sets in $\mathbb{D}$ which do not intersect the zero set of $u$. As in the proof of Lemma \ref{Lemma:main}, we can apply the maximum principle
for the ratio of $u$ and a radial positive solution $v$ to conclude that
each nodal domain $\Omega_i$ must have one of the singular points $p_k$ within its closure: $p_k \in \overline{\Omega_i}$. 

If a singular point $p_k$ has a logarithmic singularity, then it belongs to exactly one nodal domain. If $p_k$ has a pole of order $A_k-1$ and satisfies \eqref{eq: Bers pole} with $A_k-1$ in place of $k$, then at most $2(A_k-1)$ nodal domains have $p_k$ on their boundary.
Since the closure of every nodal domain must contain at least one of the points $p_k$, we conclude that the total number of nodal domains is bounded by the number of logarithmic singularities plus
$
\sum 2(A_k-1),
$
where the sum is taken over all singular points.

We have finished the sketch of the proof of Lemma \ref{Lemma:main 2}. Let's explain how it could be applied to Theorem \ref{Main 4}. We now recall the hyperbolic extension of $f=\sum_{k=1}^{n}c_{k}[(x-a_{k})^{2}+b_{k}^{2}]^{-\beta-1}$, $\beta>-1/2$, from the proof of
Theorem~\ref{Main 2}. Let $p_k:=a_k+ib_k$ be points in the upper hyperbolic plane and let $T_k$ be Mobius transformation of $\mathbb{H}$, which send $i$ to $p_k$ correspondigly. There are real coefficients $d_k$ and a solution $u$ to
\[
\Delta_{\mathbb H}u-\beta(\beta+1) u
  = -\sum_{k=1}^n d_k\,\delta_{p_k},
\]
which satisfies
\begin{equation}\label{eq:normalised-trace 2}
  u(x+iy)
    = y^{\beta+1}\bigl(f(x)+o(1)\bigr)
    \qquad\text{as }y\to0^+,
\end{equation}
where $o(1)$ is uniform in $x$. The function 
\[
u(z)
  := \sum_{k=1}^n d_k\,g_{\mathbb H}\bigl(T_k(z)\bigr),
  \qquad z\in\mathbb H\setminus\{p_1,\dots,p_n\},
\]
was constructed as a linear combination of radial fundamental solutions with logarithmic singularities at the points $p_k$.   We devoted a whole section of appendix for an explicit formula \eqref{eq: radial fundamental 2} for $g_\mathbb{H}$, which allows calculations of asymptotics for $g$.

Now, we are ready to finish the sketch of the proof of Theorem \ref{Main 4}. 
Consider the function $u_x$, which is also a solution to $\Delta u_x-\beta(\beta+1)u_x$ in $\mathbb{H}\setminus\{p_1,\dots,p_n\}$. It is not difficult to check that $u_x$ satisfies the assumptions of Lemma \ref{Lemma:main 2} and has poles of order $1$ at the points $p_k$, which yields that $u_x$ has at most $2(n-1)$ nodal domains. We skip  an asymptotic check that 
\begin{equation}
  u_x(x+iy)
    = y^{\beta+1}\bigl(f'(x)+o(1)\bigr)
    \qquad\text{as }y\to0^+.
\end{equation}
We also skip repeating the remaining argument from the proof of Theorem~\ref{Main 2} for $u_x$ in place of $u$, which shows that the zero set of $u_x$ has no more than $4n-2$ limit points.


\begin{thebibliography}{99}


\bibitem{A} N.~Alon, \emph{Tools from higher algebra}, in \emph{Handbook of Combinatorics}, Vol.~2, R.~L.~Graham, M.~Gr\"otschel and L.~Lov\'asz (eds.), North-Holland, Amsterdam, 1995, pp.~1749--1783.

\bibitem{ADK} N.~Alon, C.~Defant, N.~Kravitz and D.~G.~Zhu, \emph{Ordering candidates via vantage points}, Combinatorica \textbf{45} (2025), no.~2, Paper No.~25, 28~pp.

\bibitem{AEH}
C.~Am{\'e}ndola, A.~Engstr{\"o}m, and C.~Haase,
\newblock \emph{Maximum number of modes of Gaussian mixtures},
\newblock Information and Inference: A Journal of the IMA, \textbf{9}(3):587--600, 2020.
\newblock doi{10.1093/imaiai/iaz013}.

\bibitem{ABK}
P.~Arathoon, G.~Ball, and M.~D.~Kvalheim, \emph{The Maxwell conjecture is false}, preprint, arXiv:2607.27197 (2026). 

\bibitem{Eremenko2015}
A.~Eremenko, \emph{Equilibrium points of the potential}, unpublished notes. Available at: \url{https://www.math.purdue.edu/~eremenko/dvi/equil2.pdf}

\bibitem{GNS} A.~Gabrielov, D.~Novikov and B.~Shapiro, \emph{Mystery of point charges}, Proc. London Math. Soc. (3) \textbf{95} (2007), no.~2, 443--472.

\bibitem{Hormander}
L.~H\"ormander, \emph{Linear Partial Differential Operators}, Berlin-Heidelberg-NewYork:
Springer, 1963.


\bibitem{Ke}
E.~W. Kennedy, 
\emph{An Elementary Proof of the Frobenius Factorization Theorem for Differential Equations}, 
Canadian Mathematical Bulletin, \textbf{17}(3) (1974), pp.~379--380. 



\bibitem{K} A.~G.~Khovanskii, \emph{Fewnomials}, Translations of Mathematical Monographs, vol.~88, Amer. Math. Soc., Providence, RI, 1991.

\bibitem{LMP} M.~Levitin, D.~Mangoubi and I.~Polterovich, \emph{Topics in spectral geometry}, Graduate Studies in Mathematics, Vol.~237, Amer. Math. Soc., Providence, RI, 2023.

\bibitem{M} J.~C.~Maxwell, \emph{A Treatise on Electricity and Magnetism}, 3rd ed., Dover Publications, New York, 1954. (Reprint of the 1891 edition.)



\bibitem{NY}
D.~Novikov and S.~Yakovenko,
\emph{Simple exponential estimate for the number of real zeros of complete abelian integrals},
Annales de l'Institut Fourier (Grenoble), \textbf{45}(4) (1995), pp.~897--927.

\bibitem{Oliveira}
Gon\c{c}alo Oliveira, Critical points of point charge potentials along lines, in preparation.





\end{thebibliography}
\end{document}